%% file: hankel.tex
\documentclass{amsart}

\usepackage{amsmath}
\usepackage{amssymb}
\usepackage{bbm}
\usepackage{enumerate}
\usepackage{appendix}
\usepackage{mathrsfs}

\numberwithin{equation}{section}

\newcommand{\R}{\mathbb{R}}
\newcommand{\N}{\mathbb{N}}

\newtheorem{lemma}{Lemma}[section]

\newtheorem{propn}[lemma]{Proposition}
\newtheorem{thm}[lemma]{Theorem}

\newtheorem{remark0}[lemma]{Remark}
\newtheorem{eg0}[lemma]{Example}

\author[K. Habermann]{Karen Habermann}
\address{University of Bonn, Hausdorff Center for Mathematics,
  Endenicher Allee 62, 53115 Bonn, Germany.}
\email{habermann@iam.uni-bonn.de}
\thanks{Research supported by the German Research Foundation DFG
  through the Hausdorff Center for Mathematics.}

\title[An explicit formula for the inverse of a factorial Hankel matrix]
{An explicit formula for the inverse of a factorial Hankel matrix}
\begin{document}
\begin{abstract}
  We consider the $n\times n$ Hankel matrix $H$ whose entries are defined
  by $H_{ij}=1/s_{i+j}$ where $s_k=(k-1)!$ and prove that $H$ is
  invertible for all $n\in\N$ by providing an explicit formula for its
  inverse matrix.
\end{abstract}

\maketitle
\thispagestyle{empty}

\section{Introduction}
\input{introduction}

\section{Combinatorial identities}
\label{comid}
\input{comid}

\section{Inverse of a factorial Hankel matrix}
\label{proof}
\input{proof}

\bibliographystyle{plain}
\bibliography{references}

\end{document}

%% file: introduction.tex
Fix $n\in\N$ and let $H$ be the $n\times n$ matrix given by, for
$i,j\in\{1,\dots,n\}$,
\begin{equation*}
  H_{ij}=\frac{1}{(i+j-1)!}\;.
\end{equation*}
This defines a Hankel matrix because the entry $H_{ij}$ depends
only on the sum $i+j$. The factorial Hankel matrix $H$ is used as a
test matrix in numerical analysis and features as
{\tt gallery(`ipjfact')} in the Matrix Computation
Toolbox~\cite{higham_tool} by Nicholas Higham, also
see~\cite{higham_guide} and \cite{higham_acc}. Our interest in
studying the matrix $H$ is due to it arising in determining the
covariance structure of an iterated Kolmogorov diffusion, that is,
a Brownian motion together with a finite number of its iterated time
integrals, cf.~\cite[Section~4.4]{model_class}. To find an
explicit expression for a diffusion bridge associated with an iterated
Kolmogorov diffusion, we need to invert its covariance matrix, which
particularly requires us to invert the matrix $H$. It is therefore of
interest, both from our point of view and for using $H$ as a test
matrix, to show that the matrix
$H$ is invertible and to obtain an explicit formula for its inverse.
We use general binomial coefficients, which are
discussed in more detail in Section~\ref{comid}.
\begin{thm}\label{thm:inverse}
  For all $n\in\N$, the inverse $M$ of the
  Hankel matrix $H$ exists and it is given by
  \begin{equation*}
    M_{ij}= (-1)^{n+i+j+1}(i-1)!\,j!\,\binom{n-1}{i-1}\binom{n+j-1}{j}
    \sum_{k=0}^{i-1}\binom{n-i+k}{j-1}\binom{n+k-1}{k}\;.
  \end{equation*}
\end{thm}
In particular, it immediately follows that all the entries of the inverse matrix
$M$ are integer-valued.
Unpublished work by Gover~\cite{gover} already
contains an explicit formula for the inverse of the factorial Hankel
matrix $H$. However, our formula differs from the formula derived by
Gover, and we employ a different proof technique. While Gover first
determines expressions for the first row and last column of the
inverse of $H$ to then use a recursive procedure by
Trench~\cite{trench} to compute the remaining entries of the inverse
matrix, we prove Theorem~\ref{thm:inverse} directly
by manipulating general binomial coefficients, and in particular without
relying on any recursive procedures.
For completeness, we add that the explicit
formula~\cite[(3.17)]{gover} leads to
\begin{equation*}
  M_{ij}=n(-1)^{n-i-j-1}\sum_{k=\max(0,i+j-1-n)}^{i-1}
    \frac{(n+i+j-k-2)!\,(n+k-1)!\,(i+j-2k-1)}
         {(i+j-k-1)!\,k!\,(n+k-i-j+1)!\,(n-k)!}\;,
\end{equation*}
which, for $m=i+j-1$ and with binomial coefficients, Gover
rewrites as
\begin{align*}
  M_{ij}=(-1)^{n-m}n(m-1)!\sum_{k=\max(0,m-n)}^{i-1}
    \binom{n+m-k-1}{n-k}\binom{n+k-1}{n+k-m}
    \left(\binom{m-1}{k}-\binom{m-1}{k-1}\right)\;.
\end{align*}
We review two combinatorial identities
in Section~\ref{comid} which we frequently use in our
manipulation of general binomial coefficients, before
we give the proof of Theorem~\ref{thm:inverse} in
Section~\ref{proof}. Throughout, we use
the convention that $\N$ denotes the positive
integers and $\N_0$ the non-negative integers.

\proof[Acknowledgement]
I would like to thank Nicholas Higham who gave me access to a version
of the report by Michael Gover, since, as far as I am aware, this work
is not publicly available, and it allowed me to compare my formula to
the results by Gover. I would also like to thank Martin Lenz who first
pointed me towards that reference, and I am grateful to John Burkardt, Geoff
Tupholme, Hannah Myers, Jennifer Rowland, Alison Cullingford, and
Anthony Byrne, and all their kind help in my quest of tracking
down the report by Michael Gover.

%% file: comid.tex
We use the notion of a general binomial coefficient which, for
$t\in\R$ and $m\in\N_0$, is defined as
\begin{equation*}
  \binom{t}{m}=\prod_{i=1}^m\frac{t+1-i}{i}
  =\frac{t(t-1)\cdots(t-m+1)}{m!}\;,
\end{equation*}
where it is understood that
\begin{equation*}
  \binom{t}{0}=1\;.
\end{equation*}
Note that if $t\in\N_0$ and $t<m$ then
\begin{equation*}
  \binom{t}{m}
  =\left(\prod_{i=1}^t\frac{t+1-i}{i}\right)\frac{t+1-(t+1)}{t+1}
   \left(\prod_{i=t+2}^m\frac{t+1-i}{i}\right) =0\;.
\end{equation*}
The first identity we frequently use in the proof of
Theorem~\ref{thm:inverse} is the reflection identity for general
binomial coefficients.
\begin{propn}
  For all $t\in\R$ and $m\in\N_0$, we have
  \begin{equation*}
    \binom{t}{m}=(-1)^m\binom{m-t-1}{m}\;.
  \end{equation*}
\end{propn}
\begin{proof}
  By using the definition of a general binomial coefficient, we deduce
  \begin{align*}
    (-1)^m\binom{m-t-1}{m}
    &=(-1)^m\prod_{i=1}^m\frac{m-t-i}{i}
      =(-1)^m\prod_{j=1}^m\frac{m-t-m-1+j}{m+1-j}\\
    &=(-1)^m\prod_{j=1}^m\frac{-t-1+j}{j}
      =\prod_{j=1}^m\frac{t+1-j}{j}=\binom{t}{m}\;,
  \end{align*}
  as claimed.
\end{proof}
Secondly, we make use of the Chu-Vandermonde identity,
e.g. see~\cite[Chapter~3]{koepf}. For completeness, its statement and
a proof are given below.
\begin{propn}
  For all $s,t\in\R$ and $m\in\N_0$, we have
  \begin{equation*}
    \binom{s+t}{m}=\sum_{k=0}^m\binom{s}{k}\binom{t}{m-k}\;.
  \end{equation*}
\end{propn}
\begin{proof}
  By the general binomial theorem, we know that, for all $r\in\R$ and
  all $x\in\R$ with $|x|<1$,
  \begin{equation*}
    (1+x)^r=\sum_{m=0}^\infty\binom{r}{m}x^m\;.
  \end{equation*}
  Applying the general binomial theorem three times, we obtain that,
  for all $s,t\in\R$ and all $x\in\R$ with $|x|<1$,
  \begin{equation*}
    (1+x)^{s+t}=\sum_{m=0}^\infty\binom{s+t}{m}x^m
  \end{equation*}
  as well as 
  \begin{equation*}
    (1+x)^s(1+x)^t
    =\sum_{k=0}^\infty\binom{s}{k}x^k \sum_{l=0}^\infty\binom{t}{l}x^l\;.
  \end{equation*}
  By a discrete convolution of the two series
  \begin{equation*}
    \sum_{k=0}^\infty\binom{s}{k}x^k
    \quad\mbox{and}\quad
    \sum_{l=0}^\infty\binom{t}{l}x^l\;,
  \end{equation*}
  we deduce that
  \begin{equation*}
    \sum_{k=0}^\infty\binom{s}{k}x^k \sum_{l=0}^\infty\binom{t}{l}x^l
    =\sum_{m=0}^\infty\sum_{k=0}^m\binom{s}{k}x^k\binom{t}{m-k}x^{m-k}
    =\sum_{m=0}^\infty\sum_{k=0}^m\binom{s}{k}\binom{t}{m-k}x^m\;.
  \end{equation*}
  Due to $(1+x)^{s+t}=(1+x)^s(1+x)^t$, we established that
  \begin{equation*}
    \sum_{m=0}^\infty\binom{s+t}{m}x^m
    =\sum_{m=0}^\infty\sum_{k=0}^m\binom{s}{k}\binom{t}{m-k}x^m\;.
  \end{equation*}
  Since the latter holds for all $x\in\R$ with $|x|<1$, the desired
  identities follow.
\end{proof}

%% file: proof.tex
To simplify the presentation of the proof of
Theorem~\ref{thm:inverse}, we split up the analysis into two parts.
\begin{lemma}\label{lem1}
  For all $n\in\N$, we have, for $i,l\in\N$ and $k\in\N_0$ with
  $1\leq i,l\leq n$ and $0\leq k\leq i-1$,
  \begin{equation*}
    \sum_{j=1}^n(-1)^{j}\frac{j!}{(l+j-1)!}\binom{n+j-1}{j}
    \binom{n-i+k}{j-1}=
    (-1)^{n+i+k+1}\frac{(n-l)!}{(n-1)!}\binom{n}{n+l-i+k}\;.
  \end{equation*}
\end{lemma}
\begin{proof}
  We observe that, for $j\in\{1,\dots,n\}$,
  \begin{equation*}
    \frac{j!}{(l+j-1)!}\binom{n+j-1}{j}=
    \frac{j!}{(l+j-1)!}\frac{(n+j-1)!}{j!\,(n-1)!}=
    \frac{(n-l)!}{(n-1)!}\binom{n+j-1}{l+j-1}\;.
  \end{equation*}
  Moreover, the reflection identity for general binomial coefficients
  yields
  \begin{equation*}
    \binom{n+j-1}{l+j-1}=(-1)^{l+j-1}\binom{l-n-1}{l+j-1}\;.
  \end{equation*}
  Using both relations, we obtain
  \begin{align*}
    \sum_{j=1}^n(-1)^{j}\frac{j!}{(l+j-1)!}\binom{n+j-1}{j}
    \binom{n-i+k}{j-1}
    &=
    \sum_{j=1}^n(-1)^{j}\frac{(n-l)!}{(n-1)!}\binom{n+j-1}{l+j-1}
    \binom{n-i+k}{j-1}\\
    &=\sum_{j=1}^n(-1)^{l-1}\frac{(n-l)!}{(n-1)!}\binom{l-n-1}{l+j-1}
    \binom{n-i+k}{j-1}\;.
  \end{align*}
  If $j>n-i+k+1$, that is, if $n-i+k<j-1$, we have
  \begin{equation*}
    \binom{n-i+k}{j-1}=0
  \end{equation*}
  since $i\leq n$ guarantees that $n-i+k\geq 0$. From $k\leq i-1$, it
  also follows that $n-i+k+1\leq n$.
  The symmetry rule for binomial coefficients and reindexing the sum
  then give
  \begin{align*}
    \sum_{j=1}^n\binom{l-n-1}{l+j-1}\binom{n-i+k}{j-1}
    &=
    \sum_{j=1}^{n-i+k+1}\binom{l-n-1}{l+j-1}\binom{n-i+k}{n-i+k-j+1}\\
    &=
    \sum_{a=l}^{n+l-i+k}\binom{l-n-1}{a}\binom{n-i+k}{n+l-i+k-a}\;.
  \end{align*}
  By noting that for $a\in\N_0$ with $a< l$, we have
  $n-i+k<n+l-i+k-a$ and therefore,
  \begin{equation*}
    \binom{n-i+k}{n+l-i+k-a}=0\;,
  \end{equation*}
  and by applying the Chu-Vandermonde identity, we deduce that
  \begin{align*}
    \sum_{a=l}^{n+l-i+k}\binom{l-n-1}{a}\binom{n-i+k}{n+l-i+k-a}
    &=
    \sum_{a=0}^{n+l-i+k}\binom{l-n-1}{a}\binom{n-i+k}{n+l-i+k-a}\\
    &=\binom{l-i+k-1}{n+l-i+k}\;.
  \end{align*}
  Putting our conclusions together, and using the reflection identity
  for general binomial coefficients a second time, we obtain
  \begin{align*}
    \sum_{j=1}^n(-1)^{j}\frac{j!}{(l+j-1)!}\binom{n+j-1}{j}
    \binom{n-i+k}{j-1}
    &=(-1)^{l-1}\frac{(n-l)!}{(n-1)!}\binom{l-i+k-1}{n+l-i+k}\\
    &=(-1)^{n+i+k+1}\frac{(n-l)!}{(n-1)!}\binom{n}{n+l-i+k}\;,
  \end{align*}
  as claimed.
\end{proof}
Let $\delta_{il}$ denote the Kronecker delta for $i,l\in\N$.
\begin{lemma}\label{lem2}
  For all $n\in\N$ and all $i,l\in\N$ with $1\leq i,l\leq n$, we have
  \begin{equation*}
    \sum_{k=0}^{i-1}(-1)^k\binom{n}{n+l-i+k}\binom{n+k-1}{k}
    =\delta_{il}\;.
  \end{equation*}
\end{lemma}
\begin{proof}
  For $k\in\N_0$, if $k>i-l$ then $n<n+l-i+k$ and therefore,
  \begin{equation*}
    \binom{n}{n+l-i+k}=0\;.
  \end{equation*}
  In particular, if $l>i$, that is, if $0>i-l$,
  we immediately obtain
  \begin{equation*}
    \sum_{k=0}^{i-1}(-1)^k\binom{n}{n+l-i+k}\binom{n+k-1}{k}=0\;.
  \end{equation*}
  Let us now suppose that $l\leq i$.
  By the reflection identity for general binomial coefficients, we know
  \begin{equation*}
    (-1)^k\binom{n+k-1}{k}=\binom{-n}{k}\;,
  \end{equation*}
  and, by reindexing the sum, it follows that
  \begin{align*}
    \sum_{k=0}^{i-1}(-1)^k\binom{n}{n+l-i+k}\binom{n+k-1}{k}
    &=\sum_{k=0}^{i-l}\binom{n}{n+l-i+k}\binom{-n}{k}\\
    &=\sum_{b=0}^{i-l}\binom{n}{n-b}\binom{-n}{i-l-b}\;.
  \end{align*}
  Using the symmetry rule for binomial coefficients and the
  Chu-Vandermonde identity, we deduce
  \begin{equation*}
    \sum_{b=0}^{i-l}\binom{n}{n-b}\binom{-n}{i-l-b}
    =\sum_{b=0}^{i-l}\binom{n}{b}\binom{-n}{i-l-b}
    =\binom{0}{i-l}=\delta_{il}\;.
  \end{equation*}
  Thus, we established the desired identity both for $l>i$ and for
  $l\leq i$.
\end{proof}

Combining both results gives the proof of Theorem~\ref{thm:inverse}.
\begin{proof}[Proof of Theorem~\ref{thm:inverse}]
  By first applying Lemma~\ref{lem1} and then Lemma~\ref{lem2}, we
  conclude that, for all $n\in\N$ and all $i,l\in\{1,\dots,n\}$,
  \begin{align*}
    \left(M H\right)_{il}
    &=\sum_{j=1}^nM_{ij}H_{jl}\\
    &=\sum_{j=1}^n(-1)^{n+i+j+1}\frac{(i-1)!\,j!}{(l+j-1)!}
      \binom{n-1}{i-1}\binom{n+j-1}{j}
      \sum_{k=0}^{i-1}\binom{n-i+k}{j-1}\binom{n+k-1}{k}\\
    &=(-1)^{n+i+1}(i-1)!\,\binom{n-1}{i-1}
      \sum_{k=0}^{i-1}(-1)^{n+i+k+1}\frac{(n-l)!}{(n-1)!}\binom{n}{n+l-i+k}
      \binom{n+k-1}{k}\\
    &=(i-1)!\,\binom{n-1}{i-1}\frac{(n-l)!}{(n-1)!}
      \sum_{k=0}^{i-1}(-1)^{k}\binom{n}{n+l-i+k}\binom{n+k-1}{k}\\
    &=\frac{(n-l)!}{(n-i)!}\delta_{il}=\delta_{il}\;.
  \end{align*}
  Hence, $M$ is indeed the inverse matrix of the factorial Hankel
  matrix $H$.
\end{proof}